\documentclass[12pt]{article}

\usepackage{amsmath,amsthm,amssymb}

\theoremstyle{plain}
\newtheorem{theo}{Theorem}[section]
\newtheorem{prop}[theo]{Proposition}
\newtheorem{pp}[theo]{P}
\newtheorem{defi}[theo]{Definition}
\newtheorem{cor}[theo]{Corollary}
\newtheorem{lem}[theo]{Lemma}

\def\R{\mathbb{R}}

\def\N{\mathbb{N}}
\def\H{\mathbb{H}}
\def\C{\mathbb{C}}
\def\E{{\bf E}}
\def\P{{\bf P}}
\def\D{\mathbb{D}}
\def\U{\mathcal{U}}
\def\v2{\vskip2mm}
\def\n{\noindent}

\renewcommand{\Im}{{\rm Im}}
\def\hcap{{\rm hcap}}
\def\rad{{\rm rad}}
\def\diam{{\rm diam}}
\def\dist{{\rm dist}}

\def\({(\!(}
\def\){)\!)}
\def\a{\alpha}

\def\e{\varepsilon}
\def\de{\delta}
\def\ga{\gamma}
\def\k{\kappa}
\def\la{\lambda}

\def\Om{\Omega}

\def\v2{\vskip2mm}
\def\n{\noindent}

\def\1{{\bf 1}}

\def\n{\noindent}
\def\beq{\begin{eqnarray*}}
\def\eeq{\end{eqnarray*}}

\def\beqn{\begin{equation}}
\def\eeqn{\end{equation}}

\begin{document}
\title{Convergence of loop erased random walks on a planar graph to a chordal SLE(2) curve}
\author{Hiroyuki Suzuki }
\date{}
\maketitle

\begin{abstract}
In this paper we consider  the \lq natural' random walk  on a planar graph 
and  scale it by a small positive number $\delta$. Given a simply connected domain $D$ and its two boundary points $a$ and $b$,
we start the scaled walk at a vertex of the graph nearby $a$ and condition it on its exiting $D$ through a vertex nearby $b$,
and prove that the loop erasure of the conditioned walk converges, as $\de\downarrow 0$, to the chordal SLE$_{2}$ that
connects $a$ and $b$ in $D$, provided that an invariance principle is valid for both  the random walk and the dual walk of it.
Our result is an extension of one due to Dapeng Zhan \cite{Z} where the problem is considered on the square lattice.
A convergence to the radial SLE$_{2}$ has been obtained by Lawler, Schramm and Werner \cite{LSW} for the square 
and triangular lattices and by Yadin and Yehudayoff \cite{YY} for a wide class of planar graphs.
Our proof, though an adaptation of that of \cite{LSW} and  \cite{YY},
involves some new ingredients that arise  from two sources: one  for dealing with a martingale observable that is different from
that used in \cite{LSW} and  \cite{YY} and the other for estimating the harmonic measures of the random walk started 
at a boundary point of a domain.
\footnote[0]{MSC2010 subject classifications. 60F17, 60J67, 82B41.
 
~~ Key words and phrases, Loop erased random walk, Schramm-Loewner evolution.}
\end{abstract}

%%%%%%%%%%%
%%%%%%%%%
%%%%%%%%%%

\section{Introduction}   

The Schramm-Loewner evolutions driven by Brownian motion $\sqrt \kappa B(t)$ of variance $\kappa$,
abbreviated  as SLE$_\kappa$,
introduced by Oded Schramm \cite{S},  have been studied from various points of view and are now recognized to  well describe 
the scaling limits of certain lattice models of both physical and mathematical  interest.
Lawler, Schramm and Werner  \cite{LSW} have proved that the scaling limit of a  loop erased  random walk 
(or loop erasure (for the definition, see p.9)  of random walk, abbreviated as LERW) 
on either of the square and triangular lattices is the radial SLE$_2$.
Dapeng Zhan \cite{Z} have studied  LERW's on the square lattice but  in a multiply connected domain 
and derived the convergence of them.
In the case of a simply connected domain in particular, he  has proved the convergence to the chordal SLE$_2$.
Yadin and Yehudayoff \cite{YY} extend the result of   \cite{LSW},  the convergence of LERW to a radial SLE to that
for the natural random walks on planar graphs under  a natural setting of the problem. In this paper we consider the LERW
in a similar setting to \cite{YY} and show that LERW conditioned to connecting two boundary points
in a simply connected domain converges to a chordal SLE$_2$ curve. 

Here we state our result in an  informal way by using the terminology familiar in the theory of SLE  of which we shall  give a brief
exposition  in the next section. Let $V$ be the set of vertices of a planar graph on which a random walk (of discrete time)  is
defined and supposed to satisfy invariance principle in the sense that the linear interpolation of its space-scaled trajectory
converges to that of Brownian motion (in a topology where two curves are identified if they agree by some change of time
parameter).
For each $\delta >0$  we make the scale change of the space by $\delta$ : 
$V_\delta = \{\delta v: v\in V\}$, the set of scaled lattice points and
accordingly we make the $\delta$-scaling of our random walk so that  it moves on $V_\delta$.
Given a simply connected bounded domain $D$ and two distinct boundary points $a$ and $b$ of it,
let  $\gamma_\delta$ denote  the loop erasure of the  random walk scaled by $\delta$ 
that  starts a vertex $a_\delta$ of $V_\delta$ nearby $a$ and is conditioned to exit $D\cap  V_\delta$ 
through a vertex $b_\delta$  nearby  $b$  so that $\gamma_\delta$ is a random self-avoiding path on $D\cap V_\delta$
connecting $a_\delta$ and $b_\delta$,  which may be regarded as a \lq path'  in the planar graph.
We prove that the polygonal curve given  by linearly interpolating  
$\gamma_\delta$ converges to the chordal  SLE$_2$ curve connecting  $a$ and $b$ in  $D$ 
under a certain natural assumption on  $D$, the pair $a,b$, the planar graph and the random walk
(Theorem \ref{main2}).

For obtaining the result as stated above  we first  prove the convergence of the driving function of the loop erasure 
(Theorem \ref{main1}).
The proof is made in a way similar to \cite{LSW}, \cite{YY}  and \cite{ScSh}. 
In  \cite{ScSh}  the harmonic explorer, an evolution of a self avoiding random curve, 
is introduced and proved to converge to a chordal SLE$_{4}$ curve.
For the proof  a suitably chosen  martingale associated with the evolving random curve, called martingale observable,
plays a dominant role.  Not as in  \cite{ScSh}  we take  the  martingale observable  given by the ratio of  harmonic measures of 
a (random)  point relative to two points, the starting site of the walk and a suitably chosen site in a random domain 
defined by the loop erasure. 
This martingale is suggested in \cite{LSW} as a suitable candidate of a martingale observable 
but we need to normalize it in an appropriate way; moreover we must change the normalization as the loop erasure grows.
We apply the approximation result on the harmonic measure (Poisson kernel) proved in \cite{YY}.
To this end we need a delicate probability estimate, since our random walk starts at a boundary point and
we must deal with the conditional law given that it exits  $D\cap V_\delta$ through another boundary point. 
 
We deduce  the convergence of the loop erasure in a uniform topology from that of  the driving function 
under the hypothesis that not only the random walk but also the dual walk of it satisfy the invariance principle 
(Theorem \ref{main2}).
For the deduction we prove Proposition \ref{rev_LERW} asserting that the law of the time reversal of loop erasure of a walk 
agrees with the law of loop erasure of the time reversal  of the same walk. 

By the way, Proposition \ref{rev_LERW} provides an improvement of the convergence to a radial SLE$_2$.
In \cite{YY} the loop erasure is unti-chronological (loops are discarded in the reverse order).
The reason is that one wants to consider the loop erasure determined from the boundary.
Because the radial SLE$_2$ starts at a boundary point and stops at an inner point, 
and one wants to use a domain Markov property of the loop erasure.
In \cite{LSW}, they used the reversibility property of the loop erasure of a simple random walk proved by Lawler \cite{Law2}.
Proposition \ref{rev_LERW}  implies that the  convergence to SLE$_2$ in  the result of Yadin and Yehudayoff  is valid also 
for LERW with  the loops discarded in the chronological order instead of unti-chronological order.

The rest of the paper is organized as follows. In Sections 2 and 3 we give brief expositions  of the Loewner evolution and  
SLE, respectively,  and  the fundamental results  relevant to  the present issue or  used in the proof of our results.
In Section 4, consisting of three subsections,
we first give the framework  of our problem, the  planar graph as well as  the random walk on it,
and  bring in the LERW  together with  results associated with it (Subsection 4.1);
we then present a  martingale associated with the LERW (Subsection 4.2);
we also present the result of \cite{YY} which asserts an approximation of the harmonic measure of our random walk
by the classical Poisson kernel and a trivial lemma of the planar graph  (Subsection 4.3).
The statement and proof of  the main result of the present paper are given in Section 5.
The convergence of the loop erasure to SLE$_2$ curve with respect to the driving function is given in Subsection 5.1,
where a certain probability estimate proved in Section 6 is taken for granted.  
The convergence of the loop erasure to SLE$_2$ curve in a uniform topology  is given in Subsection 5.2,
where we prove the invariance of law of LERW  in (a double) time reversion.
In Section 6 we verify the aforementioned probability estimate which  plays an crucial role in the proof of our result,
a probability estimate of the scaled random walk on $D\cap  V_\delta$ starting at a boundary vertex  
under the conditional law given that it exists the domain  through another boundary vertex that is specified in advance.

\section{Loewner chain}
In this section, consisting of four subsections,
we give a brief exposition of the Loewner evolution and some results  relevant to  the present issue.
The standard results in the theory as given in Lawler's book \cite{Law1} are stated 
under the heading as {\bf P 2.k} ($k=1,2,\dots$).
 
\subsection{Conformal map and half-plane capacity}
Let $\H:=\{z \in \C : \Im\,z>0 \}$ be the upper half plane.
A bounded subset $A \subset \H$ is called a compact $\H$-hull 
if $A=\overline{A}\cap \H$ and $\H \setminus A$ is a simply connected domain.
Let $\mathcal{Q}$ denote the set of compact $\H$-hulls.
For any $A \in \mathcal{Q}$, there exists a unique conformal map 
$g_A:\H \setminus A \rightarrow \H$ satisfying $|g_A(z)-z|\rightarrow 0$ as 
$z\rightarrow \infty$. 
The half-plane capacity $\hcap(A)$ is defined by 
\[ \hcap(A):=\lim _{z \rightarrow \infty} z(g_A(z)-z). \]
Then, $g_A$ has the expansion
\[g_A(z)=z+\frac{ \hcap(A)}{z} +O\left( \frac{1}{|z|^2}\right) ,
 \quad z\rightarrow \infty. \]

The half-plane capacity has some nice properties, of which we need the following.

\begin{pp}(p69-71)\label{hcap}
If $r>0,x \in \R, A\in \mathcal{Q}$, then
\[
\hcap(rA)=r^2\hcap(A),\quad \hcap(A+x)=\hcap(A).
\]
If $A, B\in \mathcal{Q},A\subset B$, then
\[
\hcap(B)=\hcap(A)+\hcap(g_A(B\setminus A)).
\]
\end{pp}

\subsection{Chordal Loewner Chain in $\H$}
A chordal Loewner chain is the solution of a type of Loewner equation that describes the evolution of a curve growing 
from the boundary to the boundary of a domain in $\C$. 
In this section we consider the special case when the domain is $\H:=\{z \in \C : \Im\,  z>0 \}$,  the upper half plane
and the curve grows from the origin to the infinity in $\H$.
Suppose that $\gamma :[0,\infty) \rightarrow \overline {\H}$ is a simple curve 
with $\gamma (0)=0,\gamma (0,\infty) \subset \H$.
Then, for each  $t\geq 0$, there exists a unique conformal map 
$g_t:\H \setminus \gamma(0,t] \rightarrow \H $ 
satisfying $|g_t(z)-z| \rightarrow 0$ as $z \rightarrow \infty $. It is noted that  $g_t$ can be continuously extended 
to the (two sided) boundary of  $\H\setminus \gamma (0,t]$ along $\gamma (0,t]$.  
If $\gamma$ is parametrized by half plane capacity (i.e., if $\displaystyle\lim_{z \to \infty}z(g_t(z)-z)=2t$),
$g_t$ satisfies the following differential equation
\begin{equation} \label{cle}
\frac{\partial }{\partial t}g_t(z)  =\frac{2}{g_t(z)-U(t)},
\quad g_0(z)=z, 
\end{equation}
where $U(t)=g_t(\gamma(t))$ and $U(\cdot)$ is a $\R$-valued continuous function (see \cite{Law1}).
We call the equation (\ref{cle}) the chordal Loewner equation and $U(\cdot)$ the driving function.

Conversely, suppose that $U(\cdot) :[0,\infty)\rightarrow \R$, a continuous function, is given in advance,
for $z \in \H$,  solve the ordinary differential equation (\ref{cle}) to obtain the solution $g_t(z)$ up to the time
$T_z:=\sup\{t>0:|g_t(z)-U(t)|>0\}$ and put $K_t:=\{z \in \H:T_z\leq t\}$.
Then for $t>0$, $g_t(z)$ is a conformal map from $\H \setminus K_t$ to $\H$.
The family $(g_t)_{t \geq 0}$ describes the evolution of hulls $(K_t)_{t \geq 0}$ corresponding to $U(\cdot)$ and growing from the boundary to $\infty$.
Therefore, we have a one-to-one correspondence between $U(\cdot )$ and $(K_t)_{t \geq 0}$.
If $U(\cdot)$ is the driving function of a simple curve $\gamma $, we can recover $\gamma $ from $U(\cdot)$
by the formula $\gamma(t)=g_t^{-1}(U(t))$ and we can write $K_t=\gamma (0,t]$.
If $U(\cdot)$ is sufficiently nice, then $(K_t)_{t \geq 0}$ is generated by a curve $\gamma$ 
with $\gamma (0) \in \R, \lim_{t\to\infty} \gamma (t) =\infty$
(i.e., for any $t \geq 0$, $\H\setminus K_t$ is the unbounded component of $\H\setminus \gamma(0,t]$).
However, there exists a continuous function $U(\cdot)$ such that
$(K_t)_{t \geq 0}$ can not be generated by a curve.
There is known a sufficient condition for $U(\cdot)$ to drive a curve as given by  
\begin{pp} (p108) \label{gen_sim}
Suppose  for some $r < \sqrt{2}$ and all $s < t$,
\[|U(t)-U(s)| \leq r\sqrt{t-s} .\]
Then $(K_t)_{t \geq 0}$ is generated by a simple curve.
\end{pp}

The family  $g_t,  t\geq 0$ is called the  (chordal) Loewner chain generated by  a curve $\gamma$ or driven by a function $U(t)$.
In summary, a simple curve $\gamma$ brings out a Loewner chain, whereby it determines the driving function $U(t)$,
and conversely a continuous function $U(t)$ with appropriate regularity
generates a curve  through the Loewner chain driven by $U(t)$.

%%%%%

\begin{prop} (Lemma 2.1. in \cite{LSW}) \label{es_hull}
There exists a constant $C>0$ such that the following holds.
Let $K_t$ be the corresponding hull for a Loewner chain driven by a continuous  function $U(t)$.
Set
\[k(t):=\sqrt{t} +\sup \{|U(s)-U(0)|:0 \leq s \leq t \}.\]
Then, for any $t>0$,
\[C^{-1}k(t) \leq \diam (K_t) \leq C k(t). \]
\end{prop}

%%%\subsection{ Chordal Loewner chain in simply connected domains}
\subsection{ Chordal Loewner chains in simply connected domains}
Let $D \subsetneq \C$ be a simply connected domain and $ \partial D$ be a set of prime ends.
If $D$ is a Jordan domain, then $\partial D$ may be identified with the topological boundary of $D$.
Let $a,b$ be distinct points on $\partial D$.
For $p \in D$, we define the inner radius of $D$ with respect to $p$,
\[ \rad _p (D):= \inf \{ |z-p| : z \not \in D \}. \]
Let $\phi : D \rightarrow \H$ be a conformal map with $\phi  (a)=0, \phi  (b)=\infty$.
Although $\phi$ is not unique, any other such map  can be written as $r\phi $ for some $r>0$.
For a simple curve  $\gamma :(0,T) \rightarrow D$  connecting  $a$ and $b$ so that  $\ga(0+)=a$ and $\ga(T-)=b$,
let $g_t$ be the Loewner chain generated by the curve  $\phi \circ \gamma : (0,T) \rightarrow \H$ and put
\[\phi_t = g_t\circ\phi , \quad t\in [0,\infty).\]
We reparametrize the curve $\gamma$ so that  the curve $\phi\circ\gamma$ in $\H$ is parametrized by half plane capacity. Denote by $(\ga(t))$ the function representing the curve in this parametrization,
  so that $2t= \hcap(\phi \circ \gamma [0,t] )$.   
The driving function $U(t)$ of the chain $g_t$ is then given by
\[U(t) = \phi_t(\gamma (t)).\]
The family of conformal maps  $\phi_t, t\geq 0$ may also be called a chordal Loewner chain (in $D$) with driving function $U(t)$.
For each $s>0$,  $\phi_s$  conformally maps $D(s):= D\setminus \ga(0, s]$ onto  $\H$  with $\phi_s(a_s) = U(s)$, 
$\phi_s(b)=\infty$, where $a_s =\ga(s)$ and  the curve $\ga^{(s)}(t) := \ga(s+t)$ connects $a_s$ and $b$ in $D(s)$.
On putting
\begin{equation} \label{U1}
 g^{(s)}_t  = g_{s+t} \circ g_s^{-1} \quad \text{and} \quad \phi_t^{(s)} = \phi_{s+t},
 \end{equation}
 substitution into  $U(s+t)=  \phi_{s+t}(\gamma (s+t))$ yields 
\begin{equation} \label{U2}
 U(t+s) = \phi_t^{(s)}(\ga^{(s)}(t)).
\end{equation}
It  follows from (\ref{U1}) that   $\phi_t^{(s)} = g_t^{(s)}\circ \phi_s$  
and  $g_t^{(s)} $ (and $\phi _t^{(s)}$) is the Loewner chain generated by  the curve  $\gamma^{(s)}$;
and also, from (\ref{U2})  that $U^{(s)}(t) := U(s+t)$ is the driving function of the chain  $\phi _t^{(s)}$ in $D(s)$. 
    
Define $p(t)\in D$ by
\[ \phi_t(p(t)) = U(t)+i. \]
$p(t)$ serves as  a reference point for the study of the conformal map $\phi _t$.
(See Proposition \ref{est_mart} and the remark advanced before Lemma \ref{fund_lem}.)
\begin{lem}\label{refpt}  
Let $T>1$ and $\epsilon >0$, and, given a  pair $ (D, \gamma)$, put $\tilde T:= \sup\{t\in [0,T]: |U(t)|<1/\epsilon\}$. 
Then there exists a constant $c(T,\epsilon)>0$, 
which may also depend on  $ (D, \gamma (0))$ but does not on $(\gamma(t), t>0)$, such that
 \[
 \rad_{p(t)}(D(t)) \geq c(T,\epsilon) \,\rad_{p(0)}(D) \quad  \text{for} \quad t<\tilde T.
 \]
\end{lem}  

\begin{proof}
We claim that 
\begin{equation} \label{claim1}
 |\phi(p(t)) - \phi(\gamma (t'))| \geq 2^{-1} e^{-4\tilde T} \quad \text{if} \quad  t'\leq t  <\tilde T.
\end{equation} 
Let $ t'\leq t  <\tilde T$ and  $z=\phi(\ga(t'))$,  and put
\[ y(s)= g_s(\phi(p(t)))- g_s(z),\qquad 0 \leq s \leq t.\] 
We prove  $|y(0)|=|\phi(p(t)) - z| \geq 2^{-1} e^{-4\tilde T}$. 
Recalling  that $\Im \, g_s (w)$ is decreasing in $s$  for any $w \in \H$, we see that 
\begin{equation} \label{Loew_eq1}
\Im\, g_s\circ\phi(p(t)) \geq \Im \, g_t\circ\phi(p(t)) =1 \quad \text{if} \quad s\leq t.
\end{equation}
Applying this with $s=0$ we have  $|y(0)|\geq 1/2$ if $\Im\, z\leq 1/2$. Let $\Im\, z> 1/2$ and 
define $\tau := \inf \{ t \geq 0 :\Im\, g_t (z) =1/2 \}$. Then $\tau< t'  \leq t$ (since $\Im\, g_{t'}(z)=0$) and
the Loewner equation together with  the inequality
(\ref{Loew_eq1}) shows 
\[
\left|\frac{d}{ds}y(s) \right| =  \frac{2|y(s)|}{| g_s\circ\phi(p(t))-U(s)|\cdot |g_s(z) -U(s)|} \leq  4|y(s)| \quad
 \text{for} \quad  0\leq s\leq \tau.
 \]
Hence $|y(s)|$ is absolutely continuous and satisfies
$\frac{d}{ds}|y(s)| \leq 4|y(s)|$, so that
\[ |y(\tau)| \leq |y(0)| e^{4\tau}.\]
Using  (\ref{Loew_eq1})   again we have  $1/2 \leq \Im\, y(\tau)$ so that  $1/2\leq  |y(0)|e^{4\tilde T}$,  which is the same as what we need to prove. Thus the claim (\ref{claim1}) is verified.

It is proved in \cite{SS} (the proof of Corollary 4.3) that 
the set $\{\phi(p(t)): t<\tilde T\}$ is included in a compact set of $\H$ depending only on $T$ and $\e$,
whence according to the Koebe distortion theorem $\rad_{p(t)}(D) \geq c_0(T,\epsilon)\,\rad_{p(0)}(D)$
for some  constant $c_0(T,\epsilon)>0$.
For the proof of the lemma  it therefore suffices to show that
\[
|p(t) -\gamma (t')| \geq c_1(T,\epsilon) \, \dist ( p(t), \partial D) \quad \text{for} \quad t' \leq t <\tilde T.
\]
To this end we may suppose $|p(t) -\gamma (t')| <2^{-1}\dist (p(t),\partial D)$.
Applying (\ref{claim1})  and  the distortion theorem in turn yields
\[
2^{-1}e^{-4\tilde T}\leq  |\phi(p(t))-\phi(\gamma (t')) | \leq 
16|p(t)-\ga(t')| \cdot \frac{\dist (\phi(p(t)),\R)}{\dist (p(t), \partial D)}. 
\]
We know that $\dist (\phi(p(t)),\R) \leq M$  for some constant $M=M(T, \epsilon )>0$
from the result of \cite{SS} mentioned above. Hence
$|p(t)-\ga(t')| \geq [e^{-4T}/32 M ]\,\dist (p(t), \partial D)$
as desired. 
\end{proof}

%%%%%%%%%
\subsection{Metrics on curves}
Let $\gamma, \gamma ^j(j=1,2,\dots)$ be curves which generate the Loewner chains.
Let $U(t)$ and $ U_j(t)$ be driving functions corresponding to $\gamma$ and $\gamma ^j$, respectively.
If $U_j(t)$ converges uniformly to $U(t)$  on any bounded interval,
then we will say that $\gamma ^j$ converges to $\gamma$ with respect to the driving function.

Next, we consider the metric on the space of unparametrized curves in $\C$. 
Let $f_1, f_2  :[0,1] \rightarrow \C$ be a continuous, non-locally constant functions.
If there exists a continuously increasing bijection  $\alpha :[0,1] \rightarrow [0,1]$ such that $f_2 =f_1 \circ \alpha$, 
then we will say $f_1$ and $f_2$ are the same up to reparametrization, denoted by $f_1\sim f_2$.
A unparametrized curve $\gamma$ is defined to be an equivalence class modulo $\sim$.
Let $d_*$ be the spherical metric on $\widehat{\C}$.
We define the metric on the space of unparametrized curves by
\begin{equation}\label{d_u}
 d_\U ( \gamma _1, \gamma _2):=\inf _\alpha \left [ \sup _{0 \leq t \leq 1} d_*(f_1(t), f_2\circ \alpha (t)) \right ] ,
\end{equation}
where $f_i$ any function in the equivalence class $\gamma _i$, and the infimum is taken over all reparametrizations
$\alpha$ which are continuously increasing bijections of $[0,1]$.
We often  denote by the same notation $\gamma$ a parametrized curve as well as an unparametrized curve.
Let us denote by $\gamma ^-$ the time reversal of $\gamma$.

The convergence with respect to the driving function is weaker than the convergence with respect to the metric $d_\U$.
We will consider a sufficient condition  for the convergence with respect to the metric $d_\U$
when we have the convergence with respect to the driving function.
Let $D \subsetneq \C$ be a simply connected domain and $ \partial D$ be the set of prime ends of $D$.
Let $a, b \in \partial D$ be distinct points.
Let $\phi  : D\rightarrow \H$ be a conformal map with $\phi  (a)=0, \phi  (b)=\infty$.
Let $\phi ^- : D\rightarrow \H$ be a conformal map with $\phi ^- (b)=0, \phi  ^-(a)=\infty$.

\begin{theo} (Theorem 1.2 in \cite{SS})\label{suffi_conv}
Let $\{ \gamma ^j \}$ be a sequence of simple curves travelling from $a$ to $b$ in $D$.
Suppose that there exists simple curves $\gamma$ and $\eta$ such that
$\phi \circ \gamma ^j$ converges to $\phi \circ \gamma$ with respect to the driving function
and $\phi ^- \circ \gamma ^{j-}$ converges to $\phi ^- \circ \eta$ with respect to the driving function.
Then $\gamma ^- =\eta$ and $\gamma ^j$ converges to $\gamma $ with respect to the metric $d_\U$.
\end{theo}

%%%\section{Schramm-Loewner evolutions}
\section{Schramm-Loewner evolutions}
\subsection{SLE in the upper half plane}
Let $B_t$ be a one-dimensional standard Brownian motion with $B_0=0$. 
A chordal Schramm-Loewner evolution with parameter $\kappa >0$ (abbreviated as chordal SLE$_\kappa$)
is the random family of conformal map $g_t$ obtained from the chordal Loewner equation
\begin{equation}\label{eq9}
\frac{\partial}{\partial t} g_t (z)=\frac{2}{g_t(z)-\sqrt{\kappa}B_t}, \quad g_0(z)=z \quad ( z \in \H).
\end{equation}
Let $K_t$ be an evolving (random) hull  corresponding to SLE$_\kappa $.
Because $B_t$ is not (1/2)-H\"older continuous, 
we can not use {\bf P\ref{gen_sim}}
and it is not easy to see whether $K_t$ is generated by a curve.
However, according to the following results  $K_t$ is actually generated by a curve with full probability.
\begin{pp}(p148)\label{sle_gen_curve}
With probability 1,
the limit $\gamma (t):=\lim _{z\rightarrow 0}g_t^{-1}(z+\sqrt{\kappa} B_t)$ exists for any $t \geq 0$ 
and $K_t$ is generated by the curve $\gamma $.
\end{pp}

This curve  $\gamma $ is called a chordal SLE$_\kappa$ curve in $\H$ from $0$ to $\infty $. 
The following properties of SLE$_\kappa$ curves are easily verified.
\begin{pp}(p148)\label{sle_scaling}
Suppose that  $\gamma$ is a chordal SLE$_\kappa$ curve in $\H$ and $r>0$.
Let $ \widehat{\gamma }(t):=r^{-1}\gamma (r^2t)$.
Then, $\widehat{\gamma}$ has the same distribution as $\gamma$.
\end{pp}

\begin{pp}(p147)\label{sle_Markov}
Suppose that  $\gamma$ is a chordal SLE$_\kappa$ curve in $\H$.
Let $\tau $ be a stopping time. 
Let $\widehat{\gamma }(t):=g_{\tau}(\gamma (t+\tau ))-\sqrt{\kappa } B_{\tau}$. 
Then, $\widehat{\gamma}$ has the same distribution as $\gamma$.
\end{pp}
The behaviour of a chordal SLE$_\kappa$ curve depends on the value of the parameter $\kappa$.
There is three phases in the behaviour of a chordal SLE$_\kappa$ curve.
The two phases transitions take place at the values $\kappa =4$ and $\kappa =8$.
\begin{pp}(p150-151)\label{sle_phase}
Suppose that $\gamma$ be a chordal SLE$_\kappa$ curve in $\H$.
\begin{itemize}
\item
If $0<\kappa  \leq 4$, then w.p.1, $\gamma$ is a simple curve with $\gamma (0, \infty )\subset \H$ .
\item
If $4 < \kappa <8$, then w.p.1,  $\gamma (0, \infty) \cap \H \not = \H$ and $\cup _{t>0}
\overline{K_t} =\overline{\H}$.
\item
If $\kappa \geq 8$, then w.p.1, $\gamma $ is a space-filling curve, i.e.,  $\gamma [0, \infty)=\overline{\H} $.
\end{itemize}
\end{pp}

%%\subsection{SLE in simply connected domains}
\subsection{SLE in simply connected domains}
Let $\gamma $ be a  chordal SLE$_\kappa$ curve in $\H$ from $0$ to $\infty $. As in the subsection 2.3 
let $D \subsetneq \C$ be a simply connected domain, $ \partial D$  a set of prime ends,
$a,b$ two distinct points on $\partial D$ and $\phi : D \rightarrow \H$  a conformal map with $\phi  (a)=0, \phi  (b)=\infty$.
Although $\phi$ is not unique, any other such map $\tilde \phi$ can be written as $r\phi $ for some $r>0$.
By {\bf P \ref{sle_scaling}},  $\phi ^{-1} (\gamma )$ is independent of the choice of the map up to a time change
and  we consider SLE$_\kappa$ curves in $D$ as unparametrized curves.
A chordal SLE$_\kappa$ curve in $D$ from $a$ to $b$ is defined by $\phi ^{-1} (\gamma )$.
\v2\v2
%%%%%%%%%%

The two properties stated in the next proposition, called the domain Markov property and conformal invariance, respectively,  immediately  follow  from the definition of SLE.
\begin{pp}\label{dom_Markov}
Let $\gamma$ be a chordal SLE$_\kappa$ curve in $D$ from $a$ to $b$
and  $\mu_{a,b;D}$ be a law of $\gamma$.
Let $f : D\rightarrow D^\prime$ be  a conformal map.
Then,
\[\mu_{a,b;D}(\cdot | \gamma (0,t] ) =\mu_{\gamma (t), b;D \setminus \gamma (0,t]} (\cdot ),\]
and
\[
f \circ \mu_{a,b;D}(\cdot )=\mu_{f(a),f(b);D^\prime }(\cdot).
\]
\end{pp}
%%%%%%%%%
\v2

%%%%%%%%%%

In the theory of SLE, it is easier to prove the convergence with respect to the driving function 
than in the metric $d_{\U}$.
Theorem \ref{suffi_conv} implies  the following result,
which we shall apply the following result to derive the convergence with respect to $d_{\U}$
of LERW from that of the driving function.
Let $\phi ^- : D\rightarrow \H$ be a conformal map with $\phi ^- (b)=0, \phi  ^-(a)=\infty$.
\begin{theo} (\cite{SS})\label{lem5}
Let $\{ \gamma ^j \}$ be a sequence of simple random curves travelling from $a$ to $b$ in $D$.
Let $\kappa\leq 4$, and $\gamma (a,b)$ be the chordal SLE$_\kappa$ curve in $D$ from $a$ to $b$.
$\phi \circ \gamma ^j$ and $\phi ^- \circ \gamma ^{j-}$
converge weakly to a chordal SLE$_\kappa$ curve in $\H$ with respect to the driving function.
Then $\gamma _j$ converges weakly to $\gamma (a,b)$ with respect to $d_{\mathcal U}$.
\end{theo}

The reversibility of SLE holds at least for $\k\leq 4$.
\begin{theo} (Theorem 2.1 in \cite{Z_rev})\label{rev_sle}
Let $\kappa \leq 4$ .
The time-reversal of a chordal SLE$_\kappa$ curve in $D$ from $a$ to $b$ has the same distribution as
chordal SLE$_\kappa$ curve in $D$ from $b$ to $a$.
\end{theo}

If $\kappa >8$, then SLE curve is not reversible.

%%%%%%%%%%
%%%%%%%%%%%
%%%%%%%%%
%%%%%%%%%%

%%%%%\section{Loop erased random walks}
\section{Loop erased random walks}
\subsection{Some property of LERW}
For any $u, v \in \C $, we write $[u, v]=\{(1-t)u+tv :0\leq t\leq 1\} $
for the line segment whose end points are  $u$ and $v$.
Let $V\subset \C$ be a countable subset with $0 \in V$. Let $E:V\times V\rightarrow [0, \infty)$ and
 $E=\{ (u, v) : E(u, v)>0 \} $.
We call $G=(V, E)$ a directed weighted graph.  
We assume that $\sum _ {v \in V}E(u, v) < \infty$ for every $u \in V$, and put
\[ p(u,v):=\frac{E(u, v)}{\sum _{w \in V} E(u,w)} .\]
We call $G$ that satisfies the following conditions a planar irreducible graph.
\begin{enumerate}
\item
$G$ is a planar graph.\\
 (i.e. for every distinct edges $(u,v),(u^\prime ,v^\prime ) \in E$, $[u,v]\cap [u^\prime ,v^\prime ] \in \left\lbrace \emptyset, \{u \},\{v \}\right\rbrace $.)
\item
For any compact set $K \subset \C$, the number of vertices $v \in K$ 
is finite.
\item
The Markov chain $S(\cdot)$ on $V$ with transition probability $p(u,v)$ is irreducible.\\
(i.e. for every $u,v \in V$, there exists $n \in \N$ such that $\P (S(n)=v\ |\ S(0)=u)>0$.)
\end{enumerate}
We call $S(\cdot)$ the natural random walk on $G$.
For the reminder of this paper we think that $G$ is a planar irreducible graph.

For any simply connected domain $D\subsetneq \C$,  let $V (D):=V \cap D$.
Define
\[\partial _{out} V (D):=
\{ (u,v) \in E : [u, v]\cap \partial D \not = \emptyset, u \in V(D)\} \]
and 
\[\partial _{in}V(D):=
\{(u,v) \in E : [u, v]\cap \partial D \not = \emptyset,   v \in V(D)\} .\]
The first exit time from D is defined by 
\[ \tau  _D:= 
\begin{cases}
\inf \{ n\geq 1 : (S(n-1),S(n)) \in \partial V_{out} (D)\} \quad \text{if} \quad S(0) \in V(D)\\
\inf \{ n\geq 2 : (S(n-1),S(n)) \in \partial V_{out} (D)\} \quad \text{if} \quad (S(0),S(1)) \in \partial _{in}V(D)\\
0 \quad \text{otherwise}
\end{cases}.\]
We sometimes consider the edge $(u,v) \in \partial _{out} V (D)$ as the vertex $v$,
and the edge $(u,v) \in \partial _{in} V (D)$ as the vertex $u$;
e.g., we write $S(\tau _D) \in \partial _{out} V(D)$ 
and $S(0) \in \partial _{in}V(D)$
and for a set $J \subset \partial D$,
we write $S(\tau _D) \in J$ instead of writing $[S(\tau _D-1),S(\tau _D)]\cap J \not = \emptyset$.
\v2
\n{\sc Loop erasure.}
Let $\omega =(\omega _0, \omega _1, \dots , \omega _n)$ be a finite sequence of points.
Let $s_0=\max \{ k\geq 0: \omega _0=\omega _k \}$. Inductively, we define
$ s_m=\max \{ k\geq 0: \omega _{s_{m-1}+1}=\omega _k \}$.
If $l= \min \{m\geq 0: \omega _{s_m}=\omega _n\}$, 
then the loop erasure of $\omega $ is defined by
\[ L[\omega ]=(\omega _{s_0}, \omega_{s_1},  \dots , \omega _{s_l}). \]

The time-reversal of $\omega $ is defined by
\[\omega ^-=(\omega _n, \omega _{n-1}, \dots , \omega _0).\]
It is readily recognized that the operations $L$ and $^-$ are not commutable, namely,
$L[\omega ^-]\not=L[\omega ]^-$ in general.
If the transition probability $p(u,v)$ is symmetric,  then the following result has been proved by Lawler in \cite{Law2}.
For our purpose, we prove the following result without assuming that $p(u,v)$ is symmetric.
\begin{prop}\label{rev_LERW}
Let $S(\cdot)$ be a natural random walk on $G$.
\[\P (L[(S(0), S(1), \dots , S({\tau _D}))^-]=\omega)
= \P (L[(S(0), S(1), \dots , S({\tau _D}))]^-=\omega).\]
\end{prop}
\v2\n
{\sc Remark.}\ 
Theorem \ref{rev_LERW} implies that the convergence to the radial SLE$_2$ in the result of Yadin and Yehudayoff
(Theorem1.1 in \cite{YY}) is valid
also for LERW with the loops discarded in the chronological order instead of
unti-chronological order.
\v2

\begin{proof}
Let $\omega =(\omega _0, \dots , \omega _n)$ and
$\omega _1, \dots , \omega _{n-1} \in V(D)$ be  distinct and $(\omega _{n-1},\omega _n) \in \partial _{out}V(D)$.
Our task is to show the identity
\begin{equation}
 \P (L[(S(0), \dots , S({\tau _D}))]=\omega )
= \P(L[(S(0) , \dots , S({\tau _D}))^-]=\omega^-). \label{a1} 
\end{equation}
Let $q : V \times V \rightarrow [0, 1]$.
Set
\[G_q(x;D)=1+\sum _{k = 0}^\infty 
\sum _{\omega ^\prime \subset D :\omega ^\prime _0=x,\omega ^\prime _k=x}
 q(\omega ^\prime _0 ,\omega^\prime _1) \cdots
q(\omega ^\prime_{k-1}, \omega ^\prime_k),
\]
where the inner summation is taken over all paths 
$\omega^\prime = (\omega ^\prime_0, \dots , \omega^\prime _k) $ in $D$
such that $\omega^\prime _0 =x, \omega ^\prime_k=x$.

The probability of LERW is described by the following (See \cite{Law2}).
\[
\P (L[(S(0), \dots , S({\tau _D})]=\omega) =
\prod _{j=0}^{n-1} p(\omega _j, \omega _{j+1}) G_p(\omega _j; D \setminus \{ \omega _0, \dots ,\omega _{j-1} \} )
\]
By the exchange lemma (the equation (12.2.3) in \cite{Law2}), we get 
\begin{equation}
\P (L[(S(0), \dots , S({\tau _D})]=\omega) =
\prod _{j=0}^{n-1} p(\omega _j, \omega _{j+1}) 
G_p(\omega _j; D \setminus \{ \omega _{j+1}, \dots ,\omega _{n-1} \} )\label{a2}
\end{equation}
On the other hand,
\begin{align*}
\P(L[(S(0) , \dots , S({\tau _D}))^-]=\omega^-)
&=\sum_{\omega ^\prime \subset D : L[(\omega ^\prime )^-]=\omega^-}\prod _{i=0}^{|\omega ^\prime |-1}
p(\omega ^\prime _i, \omega ^\prime _{i+1})\\
&=\sum_{\omega ^\prime \subset D : L[\omega ^\prime ]=\omega^-}\prod _{i=0}^{|\omega ^\prime |-1}
p^*(\omega ^\prime _i, \omega ^\prime _{i+1}),
\end{align*}
where $|\omega ^\prime |$  is the length of $\omega ^\prime$ and $p^*(x,y):=p(y,x)$.
This equation and decomposing $\omega ^\prime$ between its last visit to $\omega _{n-1},\dots , \omega _0$ imply that
\begin{align}
\P(L[(S(0) , \dots , S({\tau _D}))^-]=\omega^-)
&=\prod _{j=0}^{n-1}p^*(\omega _{j+1}, \omega _{j})
G_{p^*} (\omega _{j} ; D \setminus \{ \omega _{n-1}, \dots , \omega _{j+1}\})\notag \\
&=\prod _{j=0}^{n-1}p(\omega _{j}, \omega _{j+1})
G_{p^*} (\omega _{j} ; D \setminus \{ \omega _{j+1}, \dots , \omega _{n-1}\}).\label{a3}
\end{align}
Finally observe that $G_p(x;D^\prime) = G_{p^*}(x; D^\prime)$. 
Thus, (\ref{a2}) and (\ref{a3}) imply (\ref{a1}).
\end{proof}
Let $\gamma =(\gamma _0, \gamma _1, \dots , \gamma _l)$ be the loop erasure of the time-reversal 
of the natural random walk stopped on exiting $D$. 
By Proposition \ref{rev_LERW}, we may think that $\gamma $ is the time-reversal of the loop erasure.
(In Section 5, we treat $\gamma$ as the time-reversal of the loop erasure.
But in this section, we treat $\gamma$ as the loop erasure of the time-reversal
because it is more suitable to consider the following properties of $\gamma$.)

Let $D_j:=D \setminus \cup _{i=0}^{j-1}[\gamma _i, \gamma _{i+1}]$.
For any $j \in \N$,
\[n_j:=\min \{n \geq 0 : S(n)=\gamma _j \} .\]
Because the loop erasure $\gamma$ is determined from the boundary, $\gamma$ has the following Markov property.
\begin{prop}(Lemma 3.2. in \cite{LSW})\label{LE_Markov}
Conditioned on $\gamma [0,j]$, the following holds.
\begin{enumerate}
\item
$S[0, n_j]$ and $S[n_j , \tau _D]$ are independent.
\item
$\gamma [j, l] $ has the same distribution as the loop erasure of time-reversal of
the natural random walk $S[0, \tau _{D_j}]$ conditioned to exit at $\gamma _j$.
\end{enumerate}
\end{prop}

%\subsection{Martingale observable for LERW}
\subsection{Martingale observable for LERW}
Let $D\subsetneq \C$ be a simply connected domain.
Let $S^x(\cdot)$ be a natural random walk on $G$ started at $x \in V$.
Let $v_0 \in V(D) \cup \partial _{in}V(D)$ and $\gamma$ be the loop erasure of time-reversal 
of the natural random walk $S^{v_0}[0, \tau _D]$. 
Let $D_j:=D \setminus \cup _{i=0}^{j-1}[\gamma _i, \gamma _{i+1}]$.
The hitting probability $H_j(u,v)$ is defined by
\[H_j(u,v):=\P( S^u(\tau_{D_j})=v ).\]
Let $\mathcal{F}_j$ be a filtration generated by $\gamma [0, j]$.
%PROP/MARTINGALE
\begin{prop}\label{mart_LERW}
For any $w \in V(D)$, let 
\[M_j:=\frac{H_j(w, \gamma _j)}{H_j(v_0, \gamma _j)} .\]
Then, $M_j$ is a martingale with respect to $\mathcal{F}_j$.
\end{prop}
Lawler, Schramm and Werner \cite{LSW} point out that the martingale $M_j$ given above should be a possible martingale observable, although   they don't  adopt  it but a martingale formed by the Green functions of  evolving domains. They  provide a curtailed proof that $M_j$ is a martingale. Since $M_j$ plays the central role     
in this paper we give a detailed proof of this fact.
\begin{proof}
First, we consider another representation of $M_j$.
Let $\widehat{S}^x(\cdot)$ be a independent copy of $S^x(\cdot)$
and $L_x$ be the loop erasure of the time-reversal of $\widehat{S}^x[0, \tau_D]$.
We will denote by $\mathbf{Q}$ the law of $\widehat{S}$.
Fix $\gamma [0,j]$. By proposition \ref{LE_Markov}, 
\begin{align*}
\frac{\mathbf{Q}(L_w[0,j]=\gamma [0,j])}{\mathbf{Q}(L_{v_0}[0,j]=\gamma [0,j])}
&=\frac{\mathbf{Q}(\widehat{S}^w(\tau_{D_j})=\gamma _j)\mathbf{Q}(L_{\gamma_j}[0,j]=\gamma[0,j])}
{ \mathbf{Q}(\widehat{S}^{v_0}(\tau_{D_j})=\gamma _j)\mathbf{Q}(L_{\gamma_j}[0,j]=\gamma[0,j])}\\
&=\frac{H_j(w,\gamma _j)}{H_j(v_0, \gamma _j)}.
\end{align*}
Therefore, we can write
\[M_j=\frac{\mathbf{Q}(L_w[0,j]=\gamma [0,j])}{\mathbf{Q}(L_{v_0}[0,j]=\gamma [0,j])}.\]
Hence,
\begin{align*}
\E [M_{j+1}|\gamma [0,j]]&=\sum _{v \in V(D_j)} \P (\gamma _{j+1}=v|\gamma [0,j])
\cdot \frac{\mathbf{Q}(L_w[0,j]=\gamma[0,j], L_w(j+1)=v)}{\mathbf{Q}(L_{v_0}[0,j]=\gamma[0,j], L_{v_0}(j+1)=v)},
\end{align*}
and, since $\P (\gamma _{j+1}=v|\gamma [0,j])=\mathbf{Q}( L_{v_0}(j+1)=v | L_{v_0}[0,j]=\gamma[0,j])$,
the right-hand side reduces to 
\begin{align*}
\sum _{v \in V(D_j)} \frac{\mathbf{Q}(L_w[0,j]=\gamma[0,j], L_w(j+1)=v)}{\mathbf{Q}(L_{v_0}[0,j]=\gamma[0,j])}
=\frac{\mathbf{Q}(L_w[0,j]=\gamma[0,j])}{\mathbf{Q}(L_{v_0}[0,j]=\gamma[0,j])}=M_j.
\end{align*}
Thus, $M_j$ is a martingale.
\end{proof}

\subsection{Estimates of discrete harmonic measures}
For $\delta >0$, the graph $G_\delta =(V_\delta , E_\delta )$ defined by
\[ V_\delta =\{\delta u : u \in V \} , \quad E_\delta =\{(\delta u, \delta v):E(u, v)>0\}. \]
Let the Markov chain $S_\delta (\cdot) $ on $V_\delta $ be the scaling of $S(\cdot)$ by a factor of $\delta $.
We call $S_\delta (\cdot)$ the natural random walk on $G_\delta$.
Let $S_\delta^x(\cdot)$ be a natural random walk on $G_\delta$ started at $x \in V_\delta$.
Similarly, we can define $ H_j^{(\delta)}(u,v)$, $ V_\delta (D)$, $\partial _{out}V_\delta (D)$,
$\partial _{in}V_\delta (D)$.

Let $\D=\{ z\in \C: |z|<1\} $ be the unit disc.

\begin{defi}\label{inv_pr}
If the family of the random walks $S_\delta ^x$ satisfies the following condition, 
then we say that $S_\delta ^x$ satisfies invariance principle:

\noindent
For any compact set $K \subset \D$ and  $\epsilon >0$,
there is some $\delta _0 >0$ such that the following holds.
Let $Z^x$ be a two-dimensional Brownian motion started at $x$ stopped on exiting $\D$.
For any $0<\delta < \delta _0$ and $x \in K \cap V_\delta$,
there exists a coupling of  $S^x_\delta$ and $Z^x$ satisfying
\[\P(d_{\U}( S^x_\delta[0,\tau _\D] , Z^x)> \epsilon )< \epsilon .\]
\end{defi}

In view of the Skorokhod representation theorem
the above condition is equivalent to holding that $S^x_\delta  $ weakly converges to $Z^x$ uniformly for all $x \in K$.

In \cite{YY}  (Lemma 1.2) the following result is proved. 
%\begin{prop}\label{est_mart}
\begin{prop}\label{est_mart}
Suppose that $S_\delta ^x$ satisfies invariance principle.
For any positive constants $r$, $\e$ and $\eta <1$, there exists some $\delta _0 >0$ such that
for all $0<\delta <\delta_0$  the following holds.
Let $D \subset \D$, let $p \in V_\delta (D)$ be such that $\rad _p (D) \geq r$,
and let $\psi  : D\rightarrow \D$ be a conformal map with $\psi (p)=0$.
Let $y \in V_\delta (D) $ be such that $|\psi (y)|< 1-\eta$ and let $a \in \partial _{out}V_\delta (D)$.
Then,
\[\left| \frac{H_0^{(\delta )}(y,a)}{H_0^{(\delta )}(p,a)}-\frac{K_{\D}(\psi(y),\psi(a))}{K_{\D}(\psi(p),\psi(a))}\right|<\epsilon, \]
where  $K_{\D}$ stands for the Poisson kernel of $\D$.
\end{prop}

The Poisson kernel of $\H$ is given by
\[K_{\H}(u,v):=-\frac{1}{\pi}\Im \left ( \frac{1}{u-v}\right )=\frac{1}{\pi} \frac{\Im\, u}{|u-v|^2} .\]
The result above may be translated in terms of $K_{\H}$. For our purpose we apply it in a rather trivial fashion. Let 
%\begin{cor}\label{est_mart1}
\begin{cor}\label{est_mart1}
Suppose that $S_\delta ^x$ satisfies invariance principle.
For any  constants $r>0$, $\epsilon >0$, $\eta >0$ and $\la>1$, there exists some $\delta _0 >0$ such that
for all $0<\delta <\delta_0$  the following holds.
Let $D \subset \D$, let $p \in D$ be such that $\rad _p (D) \geq r$,
and let $\phi  : D\rightarrow \H$ be a conformal map with $\phi (p)=i$. 
%Let $b\in \partial D$ be such that  $\phi(b) =\infty$. 
Let $y,w \in V_\delta (D) $ be such that $\Im\, \phi (y) > \eta,\Im\, \phi (w) > \eta$ 
and $|\phi(y)|<\lambda,|\phi(w)|<\lambda$.
Then, for all $a \in \partial _{out}V_\delta (D)$
\[\left| \frac{H_0^{(\delta )}(w,a)}{H_0^{(\delta )}(y,a)}-\frac{K_{\H}(\phi(w),\phi(a))}{K_{\H}(\phi(y),\phi(a))}\right|<\epsilon.\]
\end{cor}
\begin{proof}
Let $p _\delta \in V_\delta (D)$ be a nearest point of $p$.
Applying Proposition \ref{est_mart} with $p=p_\delta$,
\begin{align*}
\frac{H_0^{(\delta )}(w, a)}{H_0^{(\delta )}(y, a)}
&=\frac{H_0^{(\delta )}(w, a)/H_0^{(\delta )}(p_\delta, a)}
{H_0^{(\delta )}(y, a)/H_0^{(\delta )}(p_\delta, a)}\\
&=\frac{K_{\D}(\psi(w),\psi(a))}{K_{\D}(\psi(y),\psi(a))}+O(\epsilon ).
\end{align*}
Because the ratio of the Poisson kernel is conformal invariance, we find
\[\frac{K_{\D}(\psi(w),\psi(a))}{K_{\D}(\psi(y),\psi(a))}
=\frac{K_{\H}(\phi(w),\phi(a))}{K_{\H}(\phi(y),\phi(a))}.\]
This completes the proof.
\end{proof}

\v2
%\begin{lem}\label{edge}
 Here we present the following trivial lemma for convenience of a later citation.
\begin{lem}\label{edge}
Suppose that  $S_\delta ^x$ satisfy invariance principle.
For any $\epsilon >0$, there exists $\delta _0$ such that for all $0< \delta < \delta _0$,
the length of edges of $G_\delta$ in $\D$ is bounded above by $\epsilon$.
\end{lem}
\begin{proof}
Suppose that this Lemma is not true.
Then, there exists  $\epsilon >0$ such that for some sufficiently small $\delta$,
there exists an edge $e$ of $G_\de$ such that the length of $e$ is bounded below by $\epsilon$.
Since $G_\delta$ is planar graph, $S_\delta ^x$ can not cross the edge $e$, so that it cannot behave as a Brownian path and   the  invariance principle fails to hold.
\end{proof}

%%%%%%%%%5
\section{Scaling limit}
%\section{Scaling limit: Convergence of the driving function}
\subsection{Convergence with respect to the driving function}
%\subsection{Convergence with respect to the driving function}

Let $D \subsetneq \C$ be a simply connected domain and 
 $a,b$ two distinct points on $\partial D$.
We say that $\partial D$ is locally analytic at $z \in \partial D$
if there exists a one-to-one analytic function $f :\D \rightarrow \C$ with $f(0)=z$ and 
$f(\D) \cap D= f(\{ w \in \D : \Im\, w >0 \})$.
Let $G=(V,E)$ be a planar irreducible graph and  
$S_\delta ^x $  a natural random walk on $G_ \delta $ started at $x$ (see Section 4 for detailed description).
Let $\Gamma _\delta ^{a,b}$ be  
a natural random walk on $G_\delta$ started at $a_\delta$
and stopped on exiting $D$ and conditioned to hit $\partial D$ at  $b_\delta$,
where $a _\delta$ is a point of $\partial _{in}V_\delta (D)$  close to $a$
and $b _\delta$ is a point of $\partial _{out}V_\delta (D)$  close to $b$
such that there exists a path on $G_\delta $ connecting $a_\delta$ and $b_\delta$ in $D$.
If $\partial D$ is locally analytic at $a$ and $b$,
we can choose  such $a _\delta$ and  $b _\delta$.
Let $\gamma _\delta^{a,b}$ be the loop erasure of  $\Gamma _\delta ^{a,b}$.

%\begin{theo} \label{main1}
\begin{theo} \label{main1}
Suppose that $S_\delta ^x$ satisfy invariance principle.
Let $D$ be a bounded simply connected domain 
and $a,b$ be distinct points on $\partial D$. 
Suppose that $\partial D$ is locally analytic at $a$ and $b$.
Let $\phi : D \rightarrow \H$ be a conformal map with $\phi (a) =0, \phi (b)=\infty$.
Then, $\phi \circ (\gamma _\delta^{b,a})^-$
 converges weakly to the chordal SLE$_2$ curve in $\H$ as $\delta \rightarrow 0$
with respect to the driving function.
\end{theo}

\v2\n
{\sc Remark.} \ 
In order to assure the uniformity of  invariance principle so imposed in Definition \ref{inv_pr}
 it suffices to suppose it only for the walk starting at a point, e.g., the origin as is shown in \cite{U}.    
%The assumption  that the boundary curve  is analytic about $b$  is used to secure an estimate of the first hitting distribution of
% our random walk starting at $b$ (see Lemma \ref{est_hit} below). It can be much relaxed  and   in fact is dispensable 
% according to \cite{U}.  
% Since these issues may be treated independently of the present theme, we assume these extra conditions
%in the present paper.
\v2

%By Proposition \ref{rev_LERW},  we may think  that
% $(\gamma _\delta^{b,a})^-$ is the loop erasure of time-reversal of $\Gamma _\delta ^{b,a}$.
%For the proof, we will estimate an increment of the driving function.
%The basic ideas of proof are taken from \cite{LSW2} and \cite{ScSh}, although we need to use the estimates of 
%hitting probabilities of the random walk $S^b_\de$ given in the subsection 4.3. 

We abbreviate $ (\gamma _\delta^{b,a})^-=\gamma =(\gamma _0, \gamma _1, \dots , \gamma _l)$.
By Proposition \ref{rev_LERW}, 
$\gamma$ has the same distribution as  the loop erasure of the time-reversal of $\Gamma _\delta^{b,a}$.
Hence, it is possible for $\gamma$ to use results in Section 4.
Let $\mathcal{F}_j$ be a filtration generated by $\gamma [0,j]$.
We may also think that  $\gamma [0,j]$ is the simple curve that is a linear interpolation.

Let $U(t)$ be a driving function of $\phi (\gamma )$ and
$g_t$ be a Loewner chain driven by $U(t)$.
Let $t_j:=\frac{1}{2} \hcap \phi (\gamma [0,j] )$ and 
$$U_j:=U(t_j),~~\phi _j:=g_{t_j} \circ \phi~~\mbox{and}~~D_j:=D \setminus \gamma [0,j].$$
Let $p_j :=\phi _j ^{-1}(i+U_j)$.
$p_j$ plays the role of a reference point, an \lq origin', of $D_j$.
In radial case, such a  point  is fixed at the origin.
But in chordal case, $p_j$ must be  moved with  $j$, so that there remains sufficient space around $p_j$ in  $D_j$,
a sequence of reducing  domains formed by encroachment of  $\ga$ into $D$. 
(Cf. \cite{ScSh}).

We use the martingale introduced in  Proposition \ref{mart_LERW},  as in \cite{LSW} and \cite{YY}.   
But we need to normalize it  appropriately. 
We denote by  $S^b_\delta$ a natural random walk on $G_\delta$ started at $b_\delta$.
Let $A:= \phi ^{-1}( [-1,1]) $  and the normalization is made   by multiplying
 $\P ( S_\delta ^{b}(\tau _{D}) \in  A)$,  which we name $M_j$:
\beqn\label{M_j}
M_j:=\frac{H_j^{(\delta)} (w, \gamma _j)}{H_j^{(\delta)} (b, \gamma _j)} H_0^{(\delta )}(b;A),
\eeqn
(for any $\delta >0$ and $w \in V_\delta (D)$), where we write $H_0^{(\delta )}(b;A):=\P ( S_\delta ^{b}(\tau _{D}) \in  A )$.

Let $D \subsetneq \C$ be a simply connected domain, $a,b$  two distinct points on $\partial D$ 
and $\phi: D\to \H$  a conformal map with  $\phi(a)=0, \phi(b)=\infty$ as before.
Let $p=\phi^{-1}(i)$. Put $\Psi (z)=(z-i)/(z+i)$.
 Define  $\psi := \Psi \circ \phi: D\to \D$, which is  a conformal map with $\psi(b)=1, \psi(p)=0,\psi (a)=-1$.
Let  $\mathcal D =\mathcal D(r, R,\eta)$ be the collection of all quadruplets $(D,a,b,p)$ such that $\rad_p(D)\geq r$ and $D\subset R\D$ and $\psi^{-1}$ has analytic extension in $\{z\in \C: |z-1|<\eta\}$.  

In the rest of this section let $r, R$ and $\eta$ be arbitrarily fixed positive constants 
and suppose  the same hypothesis  of Theorem \ref{main1} to be valid.   
We write $\mathcal D$ for $\mathcal D(r, R,\eta)$  and consider $(D,a, b, p)\in \mathcal D $.
For dealing with the martingale observable $M_j$ defined above the following lemma plays a significant role 
and  $\mathcal D(r, R,\eta)$  is introduced as a class for which the estimates  given there is valid uniformly.
%\begin{Lem}\label{est_hit}
%\begin{lem}\label{est_hit} ~ For any $\e>0$ and $T>1$ we can find a constant $\la_0$ so that  for any $\la\geq \la_0$  there exists  numbers $\de_0>0$   and $\a \in (0,1/2)$ such that if   $D(t,\la) = \phi_t^{-1}(\{z\in \H: |z| \geq 2\la\})$, then for for all  $0<\de<\de_0$ and integers $n, j$  with  $t_n<T$  it holds that  
%\[
%\P \Big[ \Im\, \phi (S^b_\delta (\tau _{D(t_n,\la)}) < \alpha \la
% \,\Big|\, S^b _\delta ( \tau _{D_n}) = \ga_n\Big]  <\e.
%\]
%\end{lem}

\begin{lem}\label{est_hit}
There exists a number $\lambda _0=\lambda _0(\eta )>1/2$ such that for  any $\e>0$ and $\lambda >\lambda _0$,
 there exists  numbers  $\delta _0>0$   and $\alpha \in (0,1/2)$  
 such that if $(D,a,b,p)\in \mathcal {D} (r,R,\eta)$,
  $0 <\delta < \delta _ 0$ and $D' = D \setminus \phi^{-1}(\{z:|z| <2\lambda \})$,  then
\begin{equation}
\P \left( \Im\, \phi (S^b_\delta (\tau _{D'})) < \alpha \lambda
\ |\   S^b _\delta ( \tau _{D}) \in \n A \right) <\epsilon, \label{e1}
\end{equation}
and, if $\diam (\phi (\gamma [0, j]) )< 1$, then
\begin{equation}
\P \left ( \Im\, \phi (S^b_\delta (\tau _{D'})) < \alpha \lambda
\  |\  S^b _\delta ( \tau _{D_j}) = \gamma _j \right ) <\epsilon. \label{e2}
\end{equation}

\end{lem}
\v2

The proof of Lemma \ref{est_hit}  is involved and postponed to the end of Section 6.

For any $\epsilon >0$,
let 
\[
m:=\min \{ j \geq 1 : t_j \geq \epsilon ^2  \text{ or }  |U_j-U_0| \geq \epsilon \}.
\]

%%%%\begin{lem}\label{fund_lem}
\begin{lem}\label{fund_lem}
There exists a constant $C>0$  and a number $\epsilon _0>0$ such that for each positive $\epsilon < \epsilon _0$,
there exists $\delta _0 >0$ such that
if $(D,a,b,p) \in \mathcal D(r,R,\eta)$ and  $0 <\delta < \delta _ 0$, then
\[
|\E [ U_m-U_0]| \leq C \epsilon ^3,
\]
and
\[
|\E [(U_m-U_0)^2-2t_m ]| \leq C \epsilon ^3.
\]
(Although $U_0=0$, we write $U_0$ in the formulae above to indicate how they show be when the starting position
 $U_0=\gamma _0$ is not mapped to the origin by $\phi$.)
\end{lem}
%\begin{proof}
\begin{proof}  This proof is broken into four steps. It consists of certain  estimations of  the harmonic functions 
that constitutes the martingale observable defined by (\ref{M_j}).
%{\it Step 1}. 
\v2\n
{\it Step 1}. In this step we derive an expression, given in (\ref{u1}) below,  of the ratio 
\[ {H_j^{(\delta)} (b, \gamma _j)}  /{H_0^{(\delta )} (b; A)}.\]
We take sufficiently small $\epsilon _0 >0$, which we need in this proof.
Given $0<\epsilon <\epsilon _0$ we take a number  $\lambda =1/\e^3$ that will be specified shortly.    
Let $ D^\prime := D \setminus \phi ^{-1} (B(U_0 , 2\la) \cap \H )$ (note that $B(U_0 , 2\la) =\{z: |z|<2\la)$).
In the following we consider for $j=0, 1, 2, \ldots$, although we apply the resulting relation only for  $j=0,m$,
\[
H_j^{(\delta )} (b, \gamma _j )
=\sum _{y \in V_\delta (D)} \P ( S^b_\delta (\tau _{D^\prime})=y, S^b _\delta ( \tau _{D_j}) = \gamma _j)
\]
We split the sum on the right-hand side  into two  parts according as  $y$ is close to the boundary of $D$ or not.
The part of those  $y$ which are  close to the boundary must be negligible. 

Proposition \ref{es_hull} and the definition of $m$ imply that $\diam (\phi  (\gamma [0,m-1])) =O(\epsilon)$.
By Lemma \ref{edge}, the harmonic measure  from $p$ of  $\gamma [m-1,m]$ in $D_{m}$ is $O(\epsilon)$
for sufficiently small $\delta>0$.
By conformal invariance of harmonic measure, the harmonic measure from $\phi_{m-1} (p) $ of 
$\phi _{m-1} (\gamma [m-1,m])$ in $\H \setminus \phi _{m-1} (\gamma [m-1,m])$ is $O(\epsilon)$.
This implies that $\diam (\phi _{m-1} (\gamma [m-1,m])) =O(\epsilon)$, and we have
\begin{equation}
\diam (\phi (\gamma [0, m]) )= O(\epsilon ). \label{est_diam}
\end{equation}
By (\ref{est_diam}) and Lemma \ref{est_hit}, we can choose $\alpha=\a(\e) < 1/2$  so that for  all sufficiently small $\delta >0$,
for $j=0,m$,
\[
\P \left( \Im\, \phi (S^b_\delta (\tau _{D^\prime})) < \alpha \lambda
\ |\  S^b _\delta ( \tau _{D_j}) = \gamma _j \right) =O(\epsilon ^{3}).
\]
This implies
\begin{align*}
\frac
{\P ( \Im\, \phi (S^b_\delta (\tau _{D^\prime})) < \alpha \la
 , S^b _\delta ( \tau _{D_j}) = \gamma _j)}
 {\P ( \Im\, \phi (S^b_\delta (\tau _{D^\prime})) \geq \alpha \la
 , S^b _\delta ( \tau _{D_j}) = \gamma _j)}
& =\frac
{\P ( \Im \,\phi (S^b_\delta (\tau _{D^\prime})) < \alpha \la
\  |\  S^b _\delta ( \tau _{D_j}) = \gamma _j)}
 {\P ( \Im\, \phi (S^b_\delta (\tau _{D^\prime})) \geq \alpha \la
\  |\  S^b _\delta ( \tau _{D_j}) = \gamma _j)} \\
 &=O(\epsilon ^3).
\end{align*}
Therefore, 
\begin{equation}
H_j^{(\delta )} (b, \gamma _j )
=(1+O(\epsilon ^3))\sum _{\substack {y \in V_\delta (D) \\ \Im\, \phi (y) \geq \alpha \la}}
 \P ( S^b_\delta (\tau _{D^\prime})=y, S^b _\delta ( \tau _{D_j}) = \gamma _j).
 \label{u10}
\end{equation}
By  strong Markov property,
\begin{align*}
\frac{\P ( S^b_\delta (\tau _{D^\prime})=y, S^b _\delta ( \tau _{D_j}) = \gamma _j)}{H_0^{(\delta )} (b; A)}
&=\frac{\P ( S^b_\delta (\tau _{D^\prime})=y)\P ( S^y _\delta ( \tau _{D_j}) = \gamma _j)}
{\P ( S^b _\delta (\tau _D) \in A)}\\
&=\frac{\P ( S^b_\delta (\tau _{D^\prime})=y)\P ( S^y _\delta (\tau _D) \in A)}
{\P ( S^b _\delta (\tau _D) \in A)}
\cdot
\frac{\P ( S^y _\delta (\tau _{D_j}) =\gamma _j)}
{\P ( S^y _\delta (\tau _D) \in A)}.
\end{align*}
Therefore, (\ref{u10}) implies
\begin{align}
\frac{H_j^{(\delta)} (b, \gamma _j)}{H_0^{(\delta )} (b; A)}
&=(1+O(\epsilon ^3))\sum _ {\substack{y \in V_\delta (D) \\ \Im\, \phi (y) \geq \alpha \la}}
\P (S_\delta ^{b} (\tau _{D^\prime})=y \ |\  S_\delta ^{b} (\tau _D) \in A)
\cdot \frac{H_j^{(\delta)} (y, \gamma _j)}{H_0^{(\delta )} (y;A)}.\label{u1}
\end{align}

%{\it Step 2}. 
\v2\n
{\it Step 2}. 
Let $w \in V_\delta $  and  $y \in V_\delta (D)$ satisfy 
\begin{equation} \label{u00}
\Im\, \phi (w) \geq \frac{1}{2} , \  |\phi (w)-U_0| \leq 3;  \quad
\Im\, \phi (y) \geq  \alpha \lambda ,  \  \lambda \leq |\phi (y)-U_0 | \leq  2\lambda. 
\end{equation}
Applying  Corollary \ref{est_mart1} to the domain $D$ with a reference point $p$,
\begin{align}
\frac{H_0^{(\delta )}(w, \gamma_0)}{H_0^{(\delta )}(y, \gamma _0)}
=\frac{\Im\, \phi (w) / |\phi (w)-U_0|^2}{\Im\, \phi (y) / |\phi (y)-U_0|^2} +O(\epsilon ^3),\label{u2}
\end{align}
and  the assumed invariance principle implies
\begin{equation}
 H_0^{(\delta )}(y;A) =\frac{1}{\pi} \int _{- 1} ^{1} \frac{\Im\, \phi (y)} {|\phi (y) - x |^2} dx+O(\e^3\a/\la)
 \label{u3}
\end{equation}
since $|\phi(y)- U_0|^2/ \Im\, \phi(y) \leq 2\la/\a$ (recall $\a/\la$ must get small together with  $\e$).
The relations (\ref{u00}), (\ref{u2}) and (\ref{u3}) together imply 
\begin{align}
\frac{H_0^{(\delta )}(w, \gamma_0)}{H_0^{(\delta )}(y, \gamma _0)} H_0^{(\delta )}(y; A) 
&=\frac{\Im\, \phi (w)}{ \pi |\phi (w)-U_0|^2}  \int _{-1}^{1}
 \frac{|\phi (y)-U_0|^2}{|\phi (y)-x|^2} dx +O( \epsilon ^3)\notag \\
&=\frac{2}{\pi}\frac{\Im\, \phi (w)}{|\phi (w)-U_0|^2}+O(\epsilon ^3). \label{u4}
\end{align}
From  (\ref{u1}) and (\ref{u4}) we infer that
\[
\frac1{M_0} = \frac{H_0^{(\delta)} (b, \gamma _0)}{H_0^{(\delta )} (b; A) H_0^{(\delta )} (w; \gamma _0)}
=(1+O(\e^3))\sum _ {\substack{y \in V_\delta (D) \\ \Im\, \phi (y) \geq \alpha \lambda}} p(y)
\bigg / \left[ \frac{2}{\pi}\frac{\Im\, \phi (w)}{|\phi (w)-U_0|^2}+O(\epsilon ^3) \right] ,
\]
where $p(y) =\P (S_\delta ^{b} (\tau _{D^\prime})=y \ |\  S_\delta ^{b} (\tau _D) \in A)$. 
In view of Lemma \ref{est_hit}, we can suppose 
\begin{equation}
\sum _ {\substack{y \in V_\delta (D) \\ \Im\, \phi (y) \geq \alpha \lambda}} p(y)
=1+O(\epsilon ^3), \label{K2}
\end{equation}
by replacing $\alpha$ by smaller one if necessary.
Since $ \Im\, \phi(w)/|\phi(w)- U_0 |^2$ is bounded by a universal constant,  we now conclude
\begin{align}
M_0&=\frac{2}{\pi}\frac{\Im\, \phi (w)}{|\phi (w)-U_0|^2}+O(\epsilon ^3) \notag \\
&=\frac{2}{\pi}\,\Im\, \left (\frac{-1}{\phi (w)-U_0} \right )+O(\epsilon^3 ) .\label{u5}
\end{align}

%{\it Step 3}. 
\v2\n
{\it Step 3}. We derive an analogous formula for  $M_m$.
Lemma \ref{es_hull} and (\ref{est_diam}) imply
\begin{equation}
 t_m=O(\epsilon ^2), \quad |U(s)-U(0)|=O(\epsilon ) \quad \text{for} \quad \forall s \in [0, t_m]. \label{u7}
\end{equation}
The Loewner equation (\ref{cle}) shows that
\begin{equation}
|g_t (z)-z| \leq t \cdot \sup _{0 \leq s \leq t} \frac{2}{|g_s  (z)-U(s)|},  \label{ineq}
\end{equation}
and, observing the imaginary part of the Loewner equation,
\begin{equation}
1\geq \frac{\Im\,  g_t(z)}{\Im\, z} 
\geq \exp \left(- t \cdot \sup _{0 \leq s \leq t} \frac{2}{|g_s(z)-U(s)|^2} \right). \label{ineq2}
\end{equation}
We also find $\frac{d}{dt} \Im\, g_t(z) \geq -2/ \Im\, g_t (z)$, and
this implies $\frac{d}{dt} (\Im\,  g_t (z) )^2 \geq -4$.
By integrating this relation over $[0, t]$, we get
$ (\Im\,  g_t (z) )^2 \geq (\Im\, z )^2 -4t$.
Since $t_m=O(\epsilon ^2)$, we have
$\Im \, g_s \circ \phi (w) \geq 1/4 $ for $0 \leq s \leq t_m$.
Therefore, (\ref{ineq}) gives
\begin{equation}
|g_s \circ \phi (w) -\phi (w) | =O(\epsilon ^2) \quad \text{for} \quad \forall s \in [0, t_m]. \label{est_w}
\end{equation}

Let $\sigma:=\inf \{ t\geq 0 : |g_t (z)-U(t)| \leq \lambda/2 \}$.
Using (\ref{ineq}), we get $|g_\sigma (z)-z| \leq 4\sigma /\lambda$ and
\[
|z-U(0)| \leq \frac{4\sigma }{\lambda } +\frac{\lambda }{2}+|U(\sigma )-U(0)|.
\]
Thus,  if $|z-U(0)|> \lambda$, then $\sigma >t_m$. This implies
$|g_s \circ \phi (y)-U(s)| \geq \lambda /2$ for $0 \leq s \leq t_m$.
Therefore,
(\ref{ineq}) and (\ref{ineq2}) lead to
\begin{equation}\label{K}
|\phi _m (y)-\phi (y)|  = O(\epsilon ^3) \quad \text{and} \quad \frac{\Im\, \phi_m(y)}{\Im\, \phi(y)} =  1 +  O(\e^3).
\end{equation}
(\ref{est_w}) and (\ref{K}) imply
\[
\Im\, \phi _m (w) \geq \frac{1}{3},\  |\phi _m(w)-U_m|\leq 4; \quad
\Im\, \phi _m(y) \geq \frac{\alpha \lambda }{2},\  \frac{\lambda} {2}  \leq |\phi _m(y)-U_m| \leq  3\lambda.
\]
and it follows from Lemma \ref{refpt} that $\rad_{p_m}(D_m) \geq r^\prime$   for some $r^\prime>0$.
Therefore, we can apply Corollary \ref{est_mart1} to the domain $D_m$ with the reference point $p_m$,
and hence the relation (\ref{u3}) implies
\begin{align*}
\frac{H_m^{(\delta )}(w, \gamma_m)}{H_m^{(\delta )}(y, \gamma _m)}H_0^{(\delta )}(y; A) 
=&\frac{\Im\, \phi _m(w)}{ \pi |\phi _m(w)-U_m|^2} 
\int _{-1}^{1}\frac{\Im\, \phi (y)}{\Im\, \phi _m (y)} \cdot \frac{|\phi _m (y) -U_m|^2}{|\phi(y)-x|^2} dx \\
& +O(\epsilon ^3 ).
\end{align*}
Thus, from (\ref{u1}), (\ref{K2})  and (\ref{K}) we get 
\begin{equation}
M_m=\frac{2}{\pi}\Im\, \left (\frac{-1}{ \phi _m(w)-U_m}\right )+O(\epsilon ^3). \label{u6}
\end{equation}

%{\it Step 4}. 
\v2\n
{\it Step 4}. 
Proposition \ref{mart_LERW} implies that $M_j$ is a martingale.
Because $m$ is a bounded stopping time, 
\begin{equation*}
 \E [M_m-M_0]=0.  
\end{equation*}
Thus, (\ref{u5}) and (\ref{u6}) lead to
\begin{equation} \label{u11}
\E \left [\Im\, \left ( \frac{1}{ \phi _m(w)-U_m}\right )
-\Im\, \left ( \frac{1}{\phi (w)-U_0} \right ) \right ]=O(\epsilon ^3).
\end{equation}
(\ref{u7}) and (\ref{est_w}) imply
\[
 \frac{1}{g_s \circ \phi (w)-U(s)}=\frac{1}{\phi (w)-U_0}+O(\epsilon)
 \quad \text{for} \quad \forall s \in [0, t_m].
\]
By integrating this relation over $[0, t_m]$,
Loewner equation and (\ref{u7}) show that
\begin{equation}
\phi _m(w)=\phi (w) + \frac{2}{\phi (w)-U_0}\cdot t_m +O(\epsilon^3 ) .\label{u8}
\end{equation}
Let $f(u,v)=1/(u-v)$.
Using (\ref{u7}) and (\ref{u8}),
we Taylor-expand $f(\phi _m (w), U_m)-f(\phi (w), U_0)$ with respect to $\phi _m (w)-\phi (w)$
and $U_m-U_0$, up to $O(\epsilon ^3)$.
Observing imaginary part of this Taylor expansion, from (\ref{u11}) and (\ref{u8})  we get
\begin{equation}
\Im\, \left ( \frac {1}{(\phi (w)-U_0)^2} \right )\E [U_m-U_0]+
\Im\, \left ( \frac {1}{(\phi (w)-U_0)^3} \right )\E [(U_m-U_0)^2-2t_m]=O(\epsilon ^3). \label{u9}
\end{equation}
Now, we consider two different choices of $w$ under the constraint $w \in V_\delta$ such that $\Im\, \phi (w) \geq \frac{1}{2}
,|\phi (w)| \leq 3$ .
By the Koebe distortion theorem we can find $w$ satisfying $\phi (w)-U_0=i+O(\epsilon^3)$.
Then, (\ref{u9}) implies 
\begin{equation}
\E [(U_m-U_0)^2-2t_m]=O(\epsilon ^3).
\end{equation}
Similarly, we can find $w$ satisfying $\phi (w)-U_0=e^{i\frac{\pi}{3}}+O(\epsilon^3)$ and we get
\begin{equation}
\E [U_m-U_0]=O(\epsilon ^3).
\end{equation}
\end{proof}

%%%%%%%%%%%%%%%
%%%%%%%%%%%%%%
%%%%%%%%%%%%%%%%
\v2
As in Subsection 2.3, let $D(t)=D\setminus \gamma [0,t]$, $\phi _t=g_t \circ \phi$ and $p(t)=\phi _t^{-1}(i+U(t))$.
\begin{lem}\label{conf}
Let $T>1$ and $\e>0$, and,  given a  quadruplet  $(D, a, b, p) \in \mathcal D$,
put $\tilde T= \sup \{ t \in [0,T]: |U(t)|<1/\epsilon \}$. 
Then, there exists $\eta_1=\eta_1(T, \epsilon )>0$ and $r_1=r_1(T, \e)>0$  
such that $(D(t), \gamma (t), b, p(t) ) \in \mathcal D(r_1, R, \eta_1)$ for all $t<\tilde {T}$.
\end{lem}

\begin{proof} 
Let $g_t^* (z):=g_t (z)-U(t)$.
Put $\Psi (z)=(z-i)/(z+i)$.
Define the conformal map $h_t: \D\setminus \psi(\gamma [0,t]) \rightarrow \D$  
by
\[ h_t (z):= \Psi \circ g_t^* \circ \Psi ^{-1} (z) .\]
Put $\psi_t (z) := h_t \circ \psi (z)$ so that  $\psi _t : D(t) \rightarrow \D$ 
is a conformal map with $ \psi_t (\gamma (t) ) =-1, \psi _t (b) =1, \psi _t (p(t))=0$.
Clearly $\partial (\D\setminus \psi(\gamma [0,t])) $ is locally analytic at $1$
and $h_t(1)=1$.
On using  the Loewner equation we infer that  $g_t^\prime (z)=1$  as  $z\to \infty$, which implies   $h_t ^\prime (1)=1$. 
Now we can choose a positive $\eta_1<\eta/4$ such that if $t<\tilde {T}$,
then $\psi (\gamma [0,t])$ does not intersect with $B:=\{z\in \C:|z-1|<4\eta_1\}$.
Thus, $h_t$ is analytically extended to $B$ for $t<\tilde {T}$,
so that  in view of Koebe's 1/4 theorem $h_t ^{-1}$ has  an  analytic extension in 
$\{ z \in \C : |z-1|<\eta_1\}$ for $t<\tilde T$.
Since $\psi_t^{-1}=\psi ^{-1} \circ h_t^{-1}$ and $\psi^{-1}$ is analytic on $B$,
$\psi _t^{-1}$ has  an  analytic extension in 
$\{ z \in \C : |z-1|<\eta_1\}$ for $t<\tilde T$.
The existence of $r_1$ is deduced from  Lemma \ref{refpt}.
Thus the assertion of the lemma has been proved. 
\end{proof}

%%{\it Proof of Theorem \ref{main1}}.  
\v2
\noindent
{\it Proof of Theorem \ref{main1}}. 
Having proved  Lemma \ref{fund_lem}  it is easy to adapt the arguments given in \cite{ScSh}.
Let $D$ be as in the theorem and take $R$ so that $D\subset R\D$.
Let $r:=\rad _p (D)$.
From our hypothesis of local analyticity of $\partial D$ at $b$,
the function  $\psi$  has an analytic extension in a neighborhood of  $b$.
Thus, we can choose $\eta >0$ such that $\psi^{-1}$ is analytic in $\{ z \in \C :|z-1|<\eta \}$,
hence $(D, a, b, p ) \in \mathcal D(r, R, \eta )$.

Let $T>1$ and $\epsilon _1 >0$  and put  $\tilde T = \sup\{t\in [0,T]:|U(t)|<1/\epsilon _1\}$. 
Let $\epsilon >0$ be small enough.
Let $m_0=0$ and define $m_n$ inductively by
\[
m_n := \min \{j > m_{n-1} : t_j- t_{m_{n-1}} \geq \epsilon^2 \text{ or } |U_j - U_{m_{n-1}}|\geq \epsilon \}.
\]
Let $N:= \max \{ n \in \N : t_{m_n} < \tilde{T} \}$.
By Lemma \ref{conf}, we can take some positive constants $r_1 $ and $\eta _1$ 
such that $(D_{m_n}, \gamma _{m_n}, b, p_{m_n}) \in \mathcal D(r_1, R, \eta _1)$ for any $n \leq N$.

By the Markov property stated in Proposition \ref{LE_Markov},  we find that 
$\gamma ^{(t_{m_n})}(\cdot )=\gamma ( t_{m_n}+ \cdot )$ is the same distribution as 
the time-reversal of the loop erasure of 
a natural random walk on $G_\delta$ started at $b_\delta$ and stopped on exiting $D_{m_n}$ 
and conditioned to hit $\partial D_{m_n}$ at $\gamma _{m_n}$.
We apply Lemma  \ref{fund_lem} with $(D_{m_n}, \gamma _{m_n}, b, p_{m_n})$ for any $n \leq N$.
Then, we deduce from the fact stated at (\ref{U2}) that 
there exists $\delta _0=\delta _0 (\epsilon, \epsilon _1, T )>0$ such that if $\delta < \delta _0$, then for any $n \leq N$
\[
\E \Big[ U_{m_{n+1}}-U_{m_n}\ \Big|\ \gamma [0, m_n] \Big] =O(\epsilon ^3) ,
\]
and
\[
\E \Big [(U_{m_{n+1}} -  U_{m_n})^2\ \Big| \ \gamma [0, m_n] \Big] 
=\E \Big[2( t_{m_{n+1}} -t_{m_n}) \ \Big|\  \gamma [0, m_n] \Big] +O(\epsilon ^3).
\]

The rest of proof of Theorem \ref{main1} is the proof that
$U(t)$ weakly converges to $\sqrt {2} B(t)$ uniformly on $[0,T]$ as $\delta \rightarrow 0$, 
where $B(t)$ is a one-dimensional standard Brownian motion with $B(0)=0$.
This proof follows from the above estimate and the Skorokhod embedding theorem as in \cite{LSW} and \cite{ScSh}.
(See Subsection 3.3 in \cite{LSW} and Corollary 4.3 in \cite{ScSh}.)
\qed
%%%%%%%%%%%
%%%%%%%%%
%%%%%%%%%%

\subsection{Convergence with respect to the metric $d_\U$}
Now, we assume that there exists an invariant measure $\pi$ for a natural random walk $S(\cdot)$ on $G$
such that $0<\pi (v) < \infty$ for any $v \in V$.
Let $p(u,v)$ be the transition probability for $S(\cdot)$.
We consider the dual walk $S^*(\cdot)$.
The transition probability of $S^*(\cdot)$, denoted by $p^*(u,v)$, is given by
\[ p^*(u,v):=\frac{\pi (v)} {\pi (u) } p(v,u) .\]
Then, the dual walk $S^*(\cdot)$ is a natural random walk on some other planar irreducible graph.
As in the case of $S(\cdot)$, we define $(S^*)_\delta ^x , (\Gamma^*)_\delta^{a,b},(\gamma^*)_\delta^{a,b}$
corresponding to $S^*(\cdot)$.
The following lemma is a relation between the time-reversal and the dual walk.
\begin{prop}\label{dual}
Suppose that there exists an invariant measure $\pi$ for a natural random walk $S(\cdot)$ on $G$
such that $0<\pi (v) < \infty$ for any $v \in V$.
Then, the time-reversal of $\Gamma _\delta^{a,b}$ has the same distribution as $(\Gamma^*)_\delta^{b,a}$.
Similarly, the time-reversal of $\gamma _\delta^{a,b}$ has the same distribution as $(\gamma^*)_\delta^{b,a}$.
\end{prop}
\begin{proof}
The first assertion immediately follows from the definition of the dual walk and the conditional probability.
In addition to the first assertion,  applying Proposition \ref{rev_LERW},
\[ (\gamma _\delta^{a,b})^-=L[\Gamma _\delta^{a,b}]^-\stackrel{d}{=}L[(\Gamma _\delta^{a,b})^-]
\stackrel{d}{=}L[(\Gamma^*)_\delta^{b,a}]=(\gamma^*)_\delta^{b,a},\]
where $\stackrel{d}{=}$ means the same distribution.
Hence, we get the second assertion.
\end{proof}
Let $\eta^{a,b}$ be a chordal SLE$_2$ curve  in $D$ from $a$ to $b$.
Recall the metric $d_\U$ defined by (\ref{d_u}) in Subsection 2.4.
\begin{theo}\label{main2}
Suppose that there exists an invariant measure $\pi$ for a natural random walk $S(\cdot)$ on $G$
such that $0<\pi (v) < \infty$ for any $v \in V$ and 
$S_\delta ^x$ and $(S^*)_\delta ^x$ satisfy invariance principle.
Let  $D$ be a bounded simply connected domain and  $a,b \in \partial D$ be distinct points.
Suppose that $\partial D$ is locally analytic at $a$ and $b$.
Then, $\gamma _\delta ^{a, b}$ converges weakly to $\eta ^{a,b}$  
as $\delta \rightarrow 0$
with respect to the metric $d_\U$.
\end{theo}
\begin{proof}
Let $\phi : D \rightarrow \H$ be a conformal map with $\phi (a) =0, \phi (b)=\infty$
and Let $\phi ^- : D \rightarrow \H$ be a conformal map with $\phi (b) =0, \phi (a)=\infty$.
Theorem \ref{main1} implies that  $\phi ^- \circ (\gamma_\delta ^{a,b})^-$
converges weakly to a chordal SLE$_2$ with respect to the driving function.
Because we also assume that $(S^*)_\delta ^x$ satisfy invariance principle,
Theorem \ref{main1} implies that  $\phi  \circ ((\gamma^*)_\delta^{b,a})^-$
converges weakly to a chordal SLE$_2$ with respect to the driving function.
By Proposition \ref{dual}, $\gamma _\delta ^{a, b}$ is the same distribution as $((\gamma^*)_\delta^{b,a})^-$.
Hence, $\phi  \circ \gamma_\delta ^{a,b}$ converges weakly to a chordal SLE$_2$ with respect to the driving function.
Therefore, Theorem \ref{lem5} completes the proof.
\end{proof}

%\subsection{An estimate of hitting probabilities of the random walk started at a boundary point }
\section{Estimates  of hitting probabilities of the random walk started at a boundary point }

In this section we prove Lemma \ref{est_hit}. To this end it is convenient to work in the disc  $\D$ instead of  $\H$.
Let $D \subsetneq \C$ be a simply connected domain and $a,b$ be distinct points on $\partial D$.
Let $\phi : D \rightarrow \H$ be a conformal map with $\phi (a)=0, \phi (b)=\infty$.
Let $p := \phi ^{-1} (i)$. Put $\Psi(z)= (z-1)/(z+1)$ and
 $\psi = \Psi\circ \phi$ so that $\psi$   is a conformal map of $D$ onto  $\D$  with $\psi(a)= -1, \psi (b)=1, \psi (p)=0$. 
Let $S_\delta ^b $ be a natural random walk on $G_\delta$ started at $b_\delta$,
where $b_\delta$ is a point of $\partial _{in} V_\delta (D)$ close to $b$.

 Recall the class $\mathcal{D}(r,R, \eta_0)$, which  is the collection of all quadruplets $(D,a,b,p)$ such that 
$\rad _p(D) \geq r $ and $D \subset R\D$ and $\psi ^{-1}$ has analytic extension in $\{ z \in \C : |z-1|<\eta_0 \}$.
Throughout this section we consider the constants $r$, $R$ and $\eta_0$ to be fixed  
and write $\mathcal{D}$ for $\mathcal{D}(r,R, \eta_0)$; also suppose that $S^x_\delta$ satisfies invariance principle.

 For $(D, a, b,p)\in \mathcal{D}$  and  $\eta< \eta_0\wedge \frac12$  put
\[ U= U_\eta= \{ z \in D : | \psi (z) - 1| <\eta \} \]
and for any number $\alpha $ from the open interval $(0,1/2)$,

\[J_\alpha= \{ z \in \partial U : \dist (\psi (z), \partial \D) < \alpha \eta , z \in D \} .\]

\begin{prop} \label{est1}
Let $U=U_\eta$ and $J_\a$ be as described above.
Then for any  $\e>0$ there exists $\de_0= \de(\e,\eta)>0$ such that for all positive  $\de<\de_0, \a<\de_0$ 
and for all $(D,a,b,p)\in \mathcal D$,
\[ \P (S_\delta ^b (\tau _U) \in J_ \alpha \ |\  S_\delta ^b (\tau _U) \in D) < \epsilon  , \]
 Here $\delta_0$  may depend on the graph $(V,E)$.
\end{prop}
 \n
{\sc Remark.}~  It is only for this proposition  that we need the condition of the analyticity about $b$.
Without that condition the estimate of the proposition  is obtained by Uchiyama \cite{U}. 
  
\begin{proof}
 This proof is an adaptation of  a part of the arguments given in \cite{U}. 
Put
\[
 C=\{z \in \partial D : \Im\, \psi (z) >0 , |\psi (z) -1| <\eta/3\} ,
 \]
 and
 \[ 
 B= \{ z \in \C : |\psi (z) -1|<\eta /3 \} \setminus \overline{U}, \quad \Omega=B \cup C \cup U.
 \]
Let
\[
C_\delta = \{ v \in V_\delta (D) : [u,v] \cap C \not = \emptyset \text{ for some } u \in V_\delta (B)\},
\]
and $v^*$ be a vertex in $C_\delta$ such that $\Im\, \psi (v^*)$ is closest to $\eta/6$ among vertexes of $C_\delta$.   

Let $L$ denote  the last time when  the  walk  $S_\delta^{v^*}$ in $\Om$ killed when it crosses the boundary $\partial \Om$ exits $B$: 
\[
L=
\begin{cases}
1+ \max \{ 0 \leq n <\tau _\Omega : S_\delta ^{v^*}(n) \in B \}\quad 
\text{if} \quad S_\delta ^{v^*}(\tau _\Omega )  \not \in \partial B\\
\infty \quad \text {if} \quad S_\delta ^{v^*} (\tau _\Omega ) \in \partial B
\end{cases}.
\]
We write $T=\tau _U$.
Putting $J_\alpha^+=J_\alpha \cap \H$ we compute $q=\P( S_\delta ^{v^*}(\tau_\Omega) \in J_\alpha ^+)$, 
the probability that the walk  exits  $\Omega$  through $J_\alpha ^+$, which we rewrite as
\[
q=\P (S_\delta ^{v^*} (T) \circ \theta _L \in J_\alpha ^+, L < \tau _\Omega ),
\]
where the shift operator $\theta _L$ acts on $T$ as well as on $S_\delta ^{v^*}$.
By employing the strong Markov property
\begin{align*}
q&= \sum _{n=0}^\infty \sum_{y\in C_\delta} 
\P (S_\delta ^{v^*} (T) \circ \theta _n \in J_\alpha ^+, L =n, S_\delta ^{v^*}(n)=y)\\
&=\sum _{n=0}^\infty \sum_{y\in C_\delta} 
\P (S_\delta ^{v^*} (T) \circ \theta _n \in J_\alpha ^+, S_\delta ^{v^*}(n)=y)\\
&=\sum _{n=0}^\infty \sum_{y\in C_\delta} 
\P ( S_\delta ^{v^*}(n)=y) \P ( S_\delta ^y (T) \in J_\alpha ^+ )
\end{align*}

The occurrence of  the event $S_\delta ^{y}(T) \in J_\alpha ^+$ for $y\in C_\delta$  entails $S_\delta ^{y} (T)  \in D$,
so that 
$\P ( S_\delta ^{y} (T) \in J_\alpha ^+ ) = \P ( S_\delta ^{y} (T) \in J_\alpha ^+ ,  S_\delta ^{y} (T)  \in D )$.
Hence, bringing in the conditional probability
\[
p(y)= \P ( S_\delta ^{y} (T) \in J_\alpha ^+ \ |\  S_\delta ^{y} (T)  \in D ),
\]
we infer that
\[
q =  \sum _{y\in C_\delta} G_{\Omega}(v^*,y) \P (S_\delta ^y (T) \in D) p(y),
\]
where $G_\Omega$ stands for the Green function of the walk killed on exiting  $\Omega$.
We have
\[ p(y)\geq p(b), \quad y\in C_\delta, \]
for, if $\ga^{b}$ denote a path joining $b_\delta$ with $J_\alpha ^+$ in $V_\delta (U)$, 
then the walk starting at $y\in C_\delta$ and conditioned  on the event $S_\delta^y (T) \in D$ 
must hit $\gamma ^b \cup J_\alpha ^+$ before existing $U$. 
Observing  the identity
\[
 \sum_{y\in C_\delta } G_{\Omega}(v^*,y) \P (S_\delta ^y (T) \in  D)= \P (S_\delta^{ v^*}(\tau _\Omega ) \in D),
 \]
we finally obtain
\[
q \geq p(b) \P (S_\delta^{ v^*}(\tau _\Omega ) \in D).
\]
This concludes $p(b)< \epsilon /2 $ 
since $\P (S_\delta^{ v^*}(\tau _\Omega ) \in D) > 1/3$ and $ q < \epsilon /6$
for all sufficiently small $\delta$ and $\alpha$.
Let $J_\alpha ^-=J_\alpha \setminus J_\alpha ^+$.
On defining $C$ with $\Im\, \psi (z) \leq 0$ in place of $\Im\, \psi (z) > 0$ we repeat the same argument
to show that $\P ( S_\delta ^{b} (T) \in J_\alpha ^- \ |\  S_\delta ^{b} (T)  \in D )<\epsilon/2$.
\end{proof}

\begin{lem}\label{est2}  
Let $A:=\phi ^{-1}([-1,1])$.
For any $\epsilon >0$, there exists $\delta _0=\de_0(\epsilon , \eta) >0$ such that the following holds.
Let $(D,a,b,p) \in \mathcal{D}$.
Then, for all $0<\delta < \delta _0 $ and $0< \alpha < \delta _0$,
\[ \P (S_\delta ^b (\tau _U) \in J_ \alpha \ |\  S_\delta ^b (\tau _D) \in A) < \epsilon. \]
\end{lem}

\begin{proof}
By the definition of the conditional probability and the strong Markov property, 
\begin{align*}
\frac{ \P (S_\delta ^b (\tau _U) \in J_ \alpha \ |\  S_\delta ^b (\tau _D) \in A)}
{ \P (S_\delta ^b (\tau _U) \not\in J_ \alpha \ |\  S_\delta ^b (\tau _D) \in A)}
&=\frac{ \P (S_\delta ^b (\tau _U) \in J_ \alpha , S_\delta ^b (\tau _D) \in A)}
{ \P (S_\delta ^b (\tau _U) \not\in J_ \alpha , S_\delta ^b (\tau _D) \in A)}\\
&=\frac{\sum _{y \in J_\alpha}\P ( S^b_\delta(\tau _U)=y) \P(S^y_\delta (\tau _D) \in A)}
{\sum _{y \not \in J_\alpha}\P ( S^b_\delta(\tau _U)=y) \P(S^y_\delta (\tau _D) \in A)}.
\end{align*}
Because we assume invariance principle,  the  hitting probability $\P(S^y_\delta (\tau _D) \in A)$ 
can be approximated by the same probability for a Brownian motion.
Because the hitting probability for a Brownian motion is conformal invariant, we can calculate 
the hitting probability on the upper half plane instead of $D$.
Therefore, we find that there exists a universal constant $C$ such that for sufficiently small $\delta$,
\[
\frac{\sup _{y \in J_\alpha}\P(S^y_\delta (\tau _D) \in A)}
{\inf _{y \not\in J_\alpha}\P(S^y_\delta (\tau _D) \in A)}
\leq C.\]
Thus, we obtain
\[\frac{ \P (S_\delta ^b (\tau _U) \in J_ \alpha \ |\  S_\delta ^b (\tau _D) \in A)}
{ \P (S_\delta ^b (\tau _U) \not\in J_ \alpha \ |\  S_\delta ^b (\tau _D) \in A)}
\leq 
C\frac{\sum _{y \in J_\alpha} \P ( S^b(\tau _U)=y)}{\sum _{y \not \in J_\alpha} \P ( S^b(\tau _U)=y)}.
\]
Because
\[\frac{\sum _{y \in J_\alpha} \P ( S^b(\tau _U)=y)}{\sum _{y \not \in J_\alpha} \P ( S^b(\tau _U)=y)}
=\frac{\P (S^b_\delta (\tau _U) \in J_\alpha \ |\  S^b _\delta (\tau _U) \in D)}
{\P (S^b_\delta (\tau _U) \not\in J_\alpha \ |\  S^b _\delta (\tau _U) \in D)},\]
Proposition \ref{est1} completes the proof.
\end{proof}

\begin{lem}\label{est3}
For any $\epsilon >0$, there exists $\delta _0=\delta _0 (\epsilon, \eta ) >0$ such that the following holds.
Let $(D,a,b,p) \in \mathcal{D}$.
Then, for all $0<\delta < \delta _0 $ and $0< \alpha < \delta _0$ ,
\[ \P (S_\delta ^b (\tau _U) \in J_ \alpha \ |\  S_\delta ^b (\tau _D) =a_\delta ) < \epsilon, \]
where $a_\delta$ is a point of $ \partial _{out} V_\delta (D)$  close to $a$.
\end{lem}

\begin{proof}
By the definition of the conditional probability, 
\begin{align*}
\frac{ \P (S_\delta ^b (\tau _U) \in J_ \alpha \ |\  S_\delta ^b (\tau _D) = a_\delta)}
{ \P (S_\delta ^b (\tau _U) \not\in J_ \alpha \ |\  S_\delta ^b (\tau _D)=a_\delta )}
&=\frac{ \P (S_\delta ^b (\tau _U) \in J_ \alpha , S_\delta ^b (\tau _D) = a_\delta)}
{ \P (S_\delta ^b (\tau _U) \not\in J_ \alpha , S_\delta ^b (\tau _D) =a_\delta)}\\
&=\frac{ \P (S_\delta ^b (\tau _U) \in J_ \alpha , S_\delta ^b (\tau _D) = a_\delta  \ |\  S^b_\delta (\tau _D) \in A)}
{ \P (S_\delta ^b (\tau _U) \not\in J_ \alpha , S_\delta ^b (\tau _D) =a_\delta  \ |\  S^b_\delta (\tau _D) \in A)}.
\end{align*}
Since the random walk conditioned on exiting $D$ through $A$ is Markovian,
the right-hand side above may be written as
\[
\frac{\sum _{y \in J_\alpha} \P(S^b_\delta (\tau _U)=y|S^b_\delta (\tau _D) \in A)
\P ( S^y_\delta (\tau _D)=a_\delta | S^y_\delta (\tau _D) \in A)}
{\sum _{y \not\in J_\alpha} \P(S^b_\delta (\tau _U)=y|S^b_\delta (\tau _D) \in A)
\P ( S^y_\delta (\tau _D)=a_\delta | S^y_\delta (\tau _D) \in A)}.
\]
By Lemma 5.8. in \cite{YY},
there exists a universal constant $C$ such that for sufficiently small $\delta$,
\[
\frac{\sup _{y \in J_\alpha } \P (S^y_\delta (\tau _D)=a_\delta \ |\  S^y_\delta (\tau _D) \in A)}
{\inf _{y \not\in J_\alpha } \P (S^y_\delta (\tau _D)=a_\delta \ |\  S^y_\delta (\tau _D) \in A)}
\leq C.
\]
Hence, we obtain
\[
\frac{ \P (S_\delta ^b (\tau _U) \in J_ \alpha \ |\  S_\delta ^b (\tau _D) = a_\delta)}
{ \P (S_\delta ^b (\tau _U) \not\in J_ \alpha \ |\  S_\delta ^b (\tau _D)=a_\delta )}
\leq C
\frac{ \P (S_\delta ^b (\tau _U) \in J_ \alpha \ |\  S_\delta ^b (\tau _D) \in A)}
{ \P (S_\delta ^b (\tau _U) \not\in J_ \alpha \ |\  S_\delta ^b (\tau _D) \in A )}
\]
Therefore, Lemma \ref{est2} completes the proof.
\end{proof}

\n
{\it Proof of Lemma \ref{est_hit}.} 
By the mapping $\Psi(z)= (z-i)/(z+i)$, the half disc $B_+ (2\lambda ):= B(U_0 , 2\lambda) \cap \H $ is 
mapped to a small disc of radius $\sim 1/2\lambda$ and centered at $1$.
For $1/2\lambda <\eta_0$,
(\ref{e1}) follows from applying  Lemma \ref{est2} with this small disc in place of $U_\eta$,
the little discrepancy between them  making  no harm. 
If $\diam (\phi (\gamma [0,j] ))<1$, 
the difference between  $B_+ (2\lambda )$ and $g_{t_j}(B_+(2\lambda))$ is insignificant for sufficiently large $\lambda$.
Hence, we also have (\ref{e2}) by  applying Lemma \ref{est3} with $(D_j, \gamma _j , b, p_j)$,
which is legitimate because of Lemma \ref{conf}.
\qed

\vskip4mm
{\bf Acknowledgments.}~ I would like  to express my sincere gratitude to
Professor K. Uchiyama: apart from   his many invaluable and helpful comments on this paper, I am especially  grateful  for his teaching me the essential idea for the  proof of Lemma  \ref{refpt} as well as for  showing his manuscript prior to its  publication and  letting me include a result therein as Proposition \ref{est1}.

%he pointing out many errors in the original manuscript as well as for making helpful suggestions to it, the latter having  stimulated   the author to improve  some of the proofs to be more transparent. 

\end{document}